\documentclass[a4paper,11pt]{amsart}

\usepackage{amsmath,amsthm,amsfonts,amssymb,bm,wasysym}
\usepackage[usenames]{color}
\usepackage{hyperref}
\usepackage{enumerate}

\usepackage[utf8]{inputenc}
\usepackage{amssymb, mathrsfs, amsfonts, amsmath}
\usepackage{mathtools} 
\usepackage{hyperref}
\usepackage{amsthm}
\usepackage{tikz}
\usepackage[english]{babel}
\usepackage{xcolor} 
\usepackage{geometry}
\usepackage{lipsum}

\numberwithin{equation}{section}

\newtheorem{theorem}{Theorem}[section]
\newtheorem{lemma}[theorem]{Lemma}
\newtheorem{proposition}[theorem]{Proposition}

\theoremstyle{definition}

\newtheorem{remark}[theorem]{Remark}

\newtheorem{question}[theorem]{Question}

\newcommand{\R}{\mathbb{R}}

\newcommand{\Z}{\mathbb{Z}}

\newcommand{\eps}{\varepsilon}
\newcommand{\mmd}{\mathrm{d}}

\def\eps{\varepsilon}

\allowdisplaybreaks

\def\Xint#1{\mathchoice
{\XXint\displaystyle\textstyle{#1}}%
{\XXint\textstyle\scriptstyle{#1}}%
{\XXint\scriptstyle\scriptscriptstyle{#1}}%
{\XXint\scriptscriptstyle\scriptscriptstyle{#1}}%
\!\int}
\def\XXint#1#2#3{{\setbox0=\hbox{$#1{#2#3}{\int}$}
\vcenter{\hbox{$#2#3$}}\kern-.5\wd0}}

\def\dashint{\Xint-}

\DeclarePairedDelimiter\floor{\lfloor}{\rfloor}

\begin{document}
\title{On the endpoint behaviour of oscillatory maximal functions}

\author{Tainara Borges, Cynthia Bortolotto and Jo\~ao P. G. Ramos}

\address{Tainara Borges and Cynthia Bortolotto: Instituto Nacional de Matem\'atica Pura e Aplicada, Estrada Dona Castorina 110, 22460-320 - Rio de Janeiro, Brazil}
\email{tainaragb@gmail.com; bortolottocyn@gmail.com}

\address{Jo\~ao P.G. Ramos: Hausdorff Center for Mathematics, Universit\"at Bonn, Endenicher Allee 60, 53115 - Bonn, Germany}
\email{jpgramos@math.uni-bonn.de}

\maketitle

\begin{abstract}Inspired by a question of Lie, we study boundedness in subspaces of $L^1(\R)$ of oscillatory maximal functions. In particular, we construct functions in $L^1(\R)$ which are never 
integrable under action of our class of maximal functions. On the other hand, we prove that these maximal functions map certain classes of spaces resembling Sobolev spaces into $L^1(\R)$ continuously under mild curvature assumptions 
on the phase $\gamma$. 
\end{abstract}

\section{Introduction} Let $f \in L^1_{loc}(\R^n).$ One of the classical results in harmonic analysis states that the Hardy--Littlewood maximal function 
\[
Mf(x) := \sup_{r>0} \dashint_{B_r(x)} |f(y)| \, \mmd y
\]
is bounded from $L^p \to L^p$, for each $p \in (1,+\infty].$ At the endpoint $p=1,$ however, the situation changes: $M$ is bounded from $L^1$ to $L^{1,\infty},$ but a simple calculation shows that, whenever 
$f \not \equiv 0,$ then 
\[
Mf(x) \gtrsim_{f} \frac{1}{|x|^n}, \text{ for sufficiently large } x \in \R^n.
\]
This shows, in particular, that there is no nontrivial subspace $X \subset L^1$ such that $M$ maps $X$ into $L^1,$ let alone boundedly. 

On the other hand, a theory of special subspaces to $L^1$ arises naturally by considering the maximal subspace of functions $f \in L^1$ such that the \emph{smooth} maximal 
function 
\[
M_{\varphi}f(x) = \sup_{t>0} | f * \varphi_t(x) |
\]
belongs to $L^1,$ where we use the notation $\varphi_t(y) = t^{-n} \varphi(y/t)$ for the $L^1-$scaling of the smooth function $\varphi \in \mathcal{S}(\R^n).$ This defines the classical \emph{Hardy space} $H^1(\R^n),$ endowed 
with the natural norm 

\[
\|f\|_{H^1(\R^n;\varphi)} = \| M_{\varphi}f\|_{L^1(\R^n)}.
\]
The classical theory of these spaces states that these norms are all equivalent, in the sense that, whenever $\|f\|_{H^1(\R^n;\varphi)} < +\infty,$ then it holds that 
$\|f\|_{H^1(\R^n; \psi)} < +\infty, \, \forall \psi \in \mathcal{S}(\R^n),$ and these quantities are uniformly comparable for $f \in H^1(\R^n).$ See, e.g., \cite[Chapter~I]{Stein1993} for further details on the subject of Hardy 
spaces. 

Therefore, it is natural to ask whether maximal functions with more regular convolution kernels can map any nontrivial subspace of $L^1$ into $L^1,$ in any suitable way. In this framework, Lie asked \cite{Lie20191,Lie20192} 
whether the \emph{oscillatory maximal function} $M_{\gamma}$ given by
\[ M_\gamma f(x)= \sup_{r>0} \frac{1}{2r} \left| \int_{-r}^{r} f(x-t)e^{i\gamma(x, t)}\mmd t \right|=\sup_{r>0} \frac{1}{2r} \left| \int_{x-r}^{x+r} f(t)e^{i\gamma(x, x-t)}\mmd t \right|,\]
for suitable measurable functions $\gamma : \R \times \R \to \R$ can have any better endpoint behaviour than the classical Hardy--Littlewood one. This is aligned with the main results in \cite{Lie20191}, where the author analyzes 
several questions in harmonic analysis related to maximal functions and Hilbert transform in the presence of curvature. 

In particular, as oscillation generally induces cancellation and, therefore, regularity, it is natural to conjecture that $M_{\gamma}$ has better behaviour in $L^1$ than the usual maximal function. The main scope of this work 
is to explore to which extent this better behaviour happens.

\subsection{Negative results} We start by investigating boundedness in the main classical endpoint spaces $L^1$ and $H^1.$ Our first main result proves that, unlike the smooth maximal function $M_{\varphi},$ the oscillatory counterpart $M_{\gamma}$ is ill-behaved 
in those spaces, mapping a function in $H^1$ into a function with infinite integral. 

\begin{theorem}\label{thm:main_negative}
Let $\gamma$ be a measurable function, and define $M_{\gamma}$ as before. Then it holds that there is $f \in H^1$ such that 
\[
M_{\gamma}f \not\in L^1,
\]
whenever one of the conditions holds
\begin{enumerate}
 \item either $\gamma(x,t) = \alpha(x)\beta(t),$ where $\alpha,\beta$ are two measurable real functions,$\alpha \in L^{\infty}(\R)$, and $\beta$ a continuous function;
 \item $\gamma(x,t) = \sum_{j=-d}^{d} c_j(x) t^j,$ for real functions $c_j \in L^{\infty}(\R)$.   
\end{enumerate}
\end{theorem}

In particular, even under the presence of curvature, there is, in general, no hope for the maximal functions $M_{\gamma}$ to map into $L^1$ even in the case when $\gamma$ oscillates. 

The idea for constructing such a counterexample consists mainly of two steps: first, we study how $M_{\gamma}$ acts on atoms $a \in H^1.$ In doing so, we must identify the set where $M_{\gamma}a$ is large. This yields already a simple proof that $M_{\gamma}$ 
cannot be bounded from $H^1$ to $L^1$ in both of the cases of Theorem \ref{thm:main_negative}. In order to construct an explicit function that is not mapped into $L^1$ under the action of $M_{\gamma},$ the strategy 
is to add translated versions of the atoms we used. This forces the action of $M_{\gamma}$ on each of them to be `disjoint', and adjust the rate of growth of the support of the atoms to a convergent sum which is transformed 
into non-convergent under the action of $M_{\gamma};$ $k (\log k)^2$ works as such a growth for our purposes. 

\subsection{Positive results} In contrast to the negative results of Theorem \ref{thm:main_negative}, we prove that there are non-trivial subspaces $X \subset L^1$ which are mapped back into $L^1.$ These spaces are defined as $C_{p,l} =\{ f  \in \mathcal{S}'(\R) \colon \|(1+|x|^{l})f\|_{1}+\|f'\|_{p} < + \infty\}.$ Just like in Lie's case, 
the existence of `non-trivial curvature' of the phase $\gamma$ plays a crucial role in this result.

\begin{theorem}\label{thm:main_positive}
Let $\gamma$ be a measurable function, and define $M_{\gamma}$ and $C_{p,l}$ as above. Then it holds that 
$$M_{\gamma}:C_{p,l}\rightarrow L^{1}(\R) \text{ is bounded},$$
that is, $$\|M_{\gamma}f\|_{1}\lesssim_{p,l,\gamma} \|(1+|x|^{l})f\|_{1}+\|f'\|_{p} \text{, for all } f \in C_{p,l},$$
whenever one of the following conditions holds: 
\begin{enumerate} 
 \item $\gamma(x,t) = \sum_{j=1}^m c_j(x) |t|^{d_j},$ for some $c_j(x) \in L^{\infty}(\R)$ for $j=1,\dots,m-1,$ $c_m,1/c_m \in L^{\infty}(\R),$ $2 \le d_1 < d_2 < \cdots < d_m,$ $d_m>2$, $1 < l < 2,$  and $1 \le p \le +\infty;$ 
 \item $\gamma(x,t) = \gamma(t) = a(x)t^2$ for some measurable function $a$ so that $a, 1/a \in L^{\infty}(\R),$ $1 < l < 2$ and $1 \le p < + \infty.$ 
\end{enumerate}
\end{theorem}

In order to prove Theorem \ref{thm:main_positive}, we must split the real line into sets where there is almost no oscillation, so that $M_{\gamma}$ might resemble $M$ `too much', and where the oscillatation is stronger, which produces decay. The key in this 
decomposition is that the oscillatory nature of our phase $\gamma$ inside the maximal function forbids the former set to be large, and the same oscillation produces decay in the complement. 

The main feature in order to obtain better decay in the case of the set where we capture oscillation is an analysis of the \emph{radius function}, i.e., a measurable choice $r:\R \to \R_{+}$ such that 
\[
M_{\gamma}f(x) \sim \left| \frac{1}{2r(x)} \int_{x-r(x)}^{x+r(x)} f(s) e^{i \gamma(x,x-s)} \, \mmd s, \right|
\]
for all $x \in \R.$ This is not a surprise, given that recent work in the theory of regularity of maximal functions carries out similar considerations. See, for instance, \cite{Kurka2015,Tanaka2002,Luiro2007,Luiro2018,LuiroMadrid2017,AldazPerezLazaro2006} and the references 
therein. 

Unfortunately, we do not seem to be able to get rid of the demand on first-order regularity in the subspace $X$ with the present proof. In order to do so, a more sophisticated argument than our crude integration by parts must be employed. We comment more on that 
in the last section of this manuscript. 

This article is organized into four sections, the first two devoted to proving Theorems \ref{thm:main_negative} and \ref{thm:main_positive}, respectively. In the last section, we discuss some generalizations, natural open questions that arise from our main theorems and connections to 
other related topics. 

\subsection{Notation} We will use the modified Vinogradov notation; i.e., we will write $A \lesssim B$ to denote the existence of an absolute constant $C>0$ such that $A \le C \cdot B,$ and analogously for $A \gtrsim B.$ Additionally, $C,C',\tilde{C}$ will in general denote an absolute 
constant whose exact value is not important to us, and may change from line to line. 

\subsection*{Acknowledgements} The authors would like to express their gratitude towards Fr\'ed\'eric Bernicot and Cristina Benea for organizing the event `Atélier d'Analyse Harmonique', in which the first two authors were 
introduced to the problem of endpoints bounds for oscillatory maximal functions. We would also like to thank Victor Lie for suggesting the problem and for helpful discussions. T.B. and C.B. acknowledge financial support from the Funda\c c\~ao de Amparo \`a Pesquisa do Estado do Rio de Janeiro (FAPERJ). J.P.G.R. acknowledges the Deutscher Akademischer Austauschdienst (DAAD) and the Hausdorff Center for Mathematics 
for financial support. 

\section{Proof of the unboundedness results} 

\subsection{Proof of Theorem \ref{thm:main_negative}, Part 1} We show first that the maximal operator $M_{\gamma}$ is not bounded from $H^{1}$ to $L^1$ when $\gamma(x,t)=\alpha(x)\beta(t)$ as in the hypotheses of the theorem. 
Let $f_n(t):= \frac{n}{2}\chi_{[0,\frac{1}{n}]}(t)-\frac{n}{2}\chi_{[-\frac{1}{n},0]}(t)$. We observe that $\|f_{n}\|_{1}=1$ and $\|f_{n}\|_{H^{1}_{at}}\leq 1$, as $f_{n}$ is an atom. 

Let $M = \|\alpha\|_{\infty},$ and choose $\delta > 0$ so that $\beta$ differs by at most $\frac{1}{10M}$ from $\beta(0)$ in $(-\delta,\delta).$ 
%
In order to prove that there is $f \in H^1(\R)$ so that $M_{\gamma}f \not\in L^1,$ we let 
$$h_k(x) = \frac{2}{k (\log(k+1))^2} f_{k (\log{(k+1)})^2}(x),$$
and consider the function 
\[
g(x) = \sum_{k \ge 1} h_k(x-c_k), 
\]
where we define the translating points $c_k = k^2.$  We start by noticing that 
\begin{align*}
M_{\gamma}g(x) & \ge \frac{1}{2|x-c_k|} \left|\int_{0}^{\frac{1}{k(\log(k+1))^2}} e^{i \alpha(x) \beta(x-c_k-t)} \, \mmd t\right| \cr 
 & \ge \frac{1}{10|x-c_k|} \cdot \frac{1}{k (\log(k+1))^2}, \cr 
\end{align*}
for all $x$ so that $ \delta > x-c_k \ge \frac{1}{k (\log(k+1))^2}.$ The first inequality above holds as the difference between consecutive terms of the sequence $\{k^2\}_{k \ge 1}$ diverges with $k,$ which implies that computing the maximal function of $g$ is, locally, the same as calculating that of $h_k(\cdot - c_k).$ The 
second inequality, on the other hand, follows by direct comparison of the oscillatory factor being integrated with the constant $e^{i\alpha(x) \cdot \beta(0)}.$ 

In order to conclude, notice that 
\[
\|g\|_{H^1} \le \sum_{k \ge 1} \frac{1}{k (\log(k+1))^2} \|f_k\|_{H^1} = \sum_{k \ge 1} \frac{1}{k (\log(k+1))^2} < + \infty.
\]
On the other hand, the considerations above imply directly that 
\begin{align*}
\int_{\R} M_{\gamma}g(x) \, \mmd x & \ge \sum_{k \ge 1} \int_{ c_k+ \frac{1}{k (\log(k+1))^2}}^{c_k + \delta} \frac{1}{4|x-c_k|} \cdot \frac{1}{k (\log(k+1))^2} \, \mmd x \cr 
 & \gtrsim \sum_{k \ge 1} \frac{\log(\delta \cdot k (\log(k+1))^2)}{k (\log(k+1))^2} \gtrsim_{\delta} \sum_{k \ge 1} \frac{1}{k \log(k+1)} = +\infty.  
\end{align*} 

\subsection{Proof of Theorem \ref{thm:main_negative}, Part 2} 

For $\beta>0$, define the atom $f_\beta (x) = \frac{1}{2\beta}\chi_{[0,\beta]}-\frac{1}{2\beta}\chi_{[-\beta,0]}$, then $\|f_{\beta}\|_{1}=\|f_{\beta}\|_{H^{1}_{at}}=1$, for all $\beta$.  We want to prove that $\|M_{\gamma}(f_{\beta})\|_{1}\rightarrow \infty$ as $\beta\rightarrow 0$, what implies that $M_{\gamma}$ is not bounded from $H^{1}$ to $L^1.$ 

Now, negative powers of $t$ can induce oscillation around the origin, and the previous argument fails to work. We will see, however, that we can still find a good range to integrate $M_{\gamma}f_{\beta}$, which increases as $\beta$ decreases, to show that we cannot have boundedness in $H^{1}$.

For all $x> 1+ \beta$, we have 
\[ M_\gamma f_\beta (x)= \sup_{ x-\beta \leq r \leq x+ \beta}\frac{1}{2r} \left| \int_{x-r}^\beta f_{\beta}(t)e^{i \gamma(x,x-t)}\mmd t \right| \geq \frac{1}{4\beta x} \left| \int_{x-\beta}^{x} e^{i\gamma(x,t)}\mmd t \right|, \]
if we let $r=x$. Then we observe that
\begin{equation}\label{beta}
    \begin{split}
     \beta =& \left| \int_{x-\beta}^{x} e^{i\gamma(x,x-\beta)}\mmd t \right| \leq \left| \int_{x-\beta}^{x} e^{i\gamma(x,t)} \mmd t\right| + \int_{x-\beta}^{x}| e^{i\gamma(x,t)} - e^{i\gamma(x,x-\beta)} | \mmd t \\
     &\leq \left| \int_{x-\beta}^{x} e^{i\gamma(x,t)} \mmd t\right| + \int_{x-\beta}^{x}| \gamma(x,t) - \gamma(x,x-\beta)| \mmd t.
    \end{split}
\end{equation} 

In order to estimate the second integral in (\ref{beta}), we notice that
\[ | \gamma(x,t) - \gamma(x,x-\beta)| \leq  \int_{x-\beta}^{t} |\partial_u \gamma(x,u)|\, \mmd u \leq \beta \sup_{x-\beta \leq u \leq t} |\partial_u \gamma(x,u)| ,\]
and, for $ x-\beta \leq u \leq t\leq x$, 
\[ | \partial_u\gamma(x,u)| = \left| \sum_{j=-d}^{d} c_j j u^{j-1} \right| \leq \sum_{j=-d}^{1} |c_j| |j| + u^{d-1}\sum_{j=2}^{d}|c_j|j\leq c u^{d-1} \leq c x^{d-1},\]
for $c = 2d\max\{\|c_j\|_{\infty}\}_{j=-d}^d$, where we used that $u\geq x-\beta\geq 1$. So,
\begin{equation*}
        \int_{x-\beta}^{x} |\gamma(x,t)-\gamma(x,x-\beta)|\mmd t\leq \int_{x-\beta}^{x} \beta c x^{d-1}\mmd t=c \beta^2 x^{d-1}
\end{equation*}
and returning to what we had in inequality (\ref{beta}), we get 

\begin{equation*}
    \begin{split}
        &\beta\leq \left|\int_{x-\beta}^{x} e^{i\gamma(x,t)}\mmd t\right| + C\beta^2 x^{d-1}\\
        \Rightarrow &\left|\int_{x-\beta}^{x} e^{i\gamma(x,t)}\mmd t\right|\geq \beta(1-C\beta x^{d-1})\geq \frac{\beta}{2}
    \end{split}
\end{equation*}
if $x\leq (\frac{1}{2C\beta})^{\frac{1}{d-1}}$, since:
$$1-C\beta x^{d-1}\geq 1/2\Leftrightarrow 1/2\geq C\beta x^{d-1}\Leftrightarrow x\leq (\frac{1}{2C\beta})^{\frac{1}{d-1}}.$$

So we conclude that, assuming $\beta$ small enough, for all $1+\beta\leq x\leq (\frac{1}{2C\beta})^{\frac{1}{d-1}}$, it holds that $M_{\gamma}f_{\beta}(x)\geq \frac{1}{4\beta x}\cdot\frac{\beta}{2}=\frac{1}{8x}$ and
\begin{equation}\label{eq lower bound}
    ||M_{\gamma}f_{\beta}||_{1}\geq \frac{1}{8} \int_{1+\beta}^{(\frac{1}{2C\beta})^{\frac{1}{d-1}} }\frac{1}{x}\mmd x= \frac{1}{8} \{\log\left((\frac{1}{2C\beta})^{\frac{1}{d-1}}\right)-\log(1+\beta)\} \ge C |\log (\beta)|,
\end{equation}
 as $\beta\rightarrow 0$. This already proves unboundedness for this case of $\gamma.$ 
 
In order to construct $f \in H^1$ such that $M_{\gamma}f \notin L^1,$ we let again $\beta_n = \frac{1}{n (\log (n+1))^2},$ and consider the function
\[
g(x) = \sum_{n \ge 1} \beta_n f_{n}(x-b_n),
\]
where we choose $b_n = 2^{2^n}$ and $f_n  = f_{\frac{1}{\beta_n}}$. As $M_{\gamma} \le M$ pointwise, we have that $\beta_n \cdot M_{\gamma}(f_{n}) \le \frac{\beta_n}{|x|},$ for $|x| \ge C_0,$ 
where $C_0$ is an universal constant. 

By the choice of $b_n$ and an argument entirely analogous to \eqref{eq lower bound}, we have that 
\[
M_{\gamma}(\tau_{b_n} f_n) \ge \frac{1}{8|x-b_n|}, 
\]
if $1 + \beta_n \le |x - b_n| \le \left(\frac{1}{2C\beta_n}\right)^{1/(d-1)},$ where $\tau_bf(x) = f(x-b)$ denotes the translation operator. An application of triangle's inequality then yields 

\begin{align*} 
M_{\gamma}g(x) & \ge  \beta_n M_{\gamma}(\tau_{b_n}f_{n})(x) - \sum_{m \ne n} \beta_m M f_{m}(x) \cr
 & \ge \frac{\beta_n}{8|x-b_n|} - \sum_{m \ne n} \frac{\beta_m}{|x-b_m|} \cr
 & \ge \frac{\beta_n}{8|x-b_n|} - \frac{C}{2^{2^n}} - \sum_{m > n} \frac{4\beta_m}{2^{2^m}} \ge \frac{C'\beta_n}{|x-b_n|},
\end{align*}
as $|x-b_m| \ge |b_m - b_n| - |x-b_n| \ge \frac{1}{4} \max \{b_m,b_n\},$ for $m,n$ sufficiently large. 
This computation allows us to estimate
\begin{align*}
\int_{\R} M_{\gamma}g(x) \, \mmd x & \ge C' \sum_{n \ge 1} \int_{1 + \beta_n \le |x - b_n| \le \left(\frac{1}{2C\beta_n}\right)^{1/(d-1)}} \frac{\beta_n}{|x-b_n|} \, \mmd x \cr 
 & \gtrsim \sum_{n \ge 1} \beta_n |\log(\beta_n)| \gtrsim \sum_{n \ge 1} \frac{1}{n \log (n+1)} = +\infty,
\end{align*}
whereas 
\begin{align*}
\|g\|_{H^1} \le \sum_{n \ge 1} \beta_n \|f_{\beta_n}\|_{H^1} = \sum_{n \ge 1} \frac{1}{n \log(n+1)^2} <+\infty,
\end{align*}
which provides us with the desired counterexample. 

\begin{remark}
Let $\gamma(t)$ be a function in $C^{2}(\mathbb{R}\backslash \{0\})$ that satisfies the following conditions:
\begin{enumerate}
\item $\gamma'(t)\neq 0$ for all $t\neq 0$.
    \item  $|\gamma'|$ is non-decreasing for $t>0$ and $|\gamma'|$ is non-increasing for $t<0$.
    \item The function $\dfrac{1}{x\gamma'(x)}$ is integrable at infinity, that is, there exists a $R>0$, such that $$\int_{|x|>R} \dfrac{1}{|x\gamma'(x)|}<\infty.$$
\end{enumerate}
Define $M_{\gamma}f(x)=\sup_{r>0}\frac{1}{2r}|\int_{x-r}^{x+r}f(t)e^{i\gamma(x-t)}\mmd t|$.
Then, every characteristic function of bounded interval is mapped in $L^1(\R)$ by $M_{\gamma}$. By sublinearity, every step function is mapped in $L^1$. In particular, this is satisfied for $\gamma(t)=t^k$ with $k\geq 2$ natural number or $\gamma(t)=|t|^k$, $k>1$ real number. 

Indeed, observe that since $\gamma'$ is a continuous function that is never zero in $(0,\infty)$ and $|\gamma'|$ is non-decreasing for $t>0$,
then $ \frac{1}{\gamma'}  $ is monotonic in this interval and $\dfrac{d}{dt}\left(\dfrac{1}{\gamma'}\right)$  does not change sign in $(0,\infty)$. 
The same holds for $(-\infty,0)$.

Second, $M_{\gamma}$ commutes with translations, as $\gamma$ does not depend on $x$. So, we can restrict ourselves to the case $f_{\beta}(t)=\chi_{[-\beta,\beta]}(t)$, some $\beta>0$.
In this case, for every $x>\beta$, we have:

\begin{equation*}
\begin{split}
    M_{\gamma}(f_{\beta})(x)= \sup_{r>0}\frac{1}{2r}\left| \int_{x-r}^{x+r} f_{\beta}(t)e^{i\gamma(x-t)}\mmd t \right|
    & \leq\frac{1}{2(x-\beta)}\sup_{x-\beta\leq r \leq x+\beta}\left| \int_{x-\beta}^{r} e^{i\gamma(t)}\mmd t \right|,
\end{split}   
\end{equation*}
and for any $0<x-\beta\leq r \leq x+\beta$,
\begin{equation*}
\begin{split}
    \left| \int_{x-\beta}^{r} e^{i\gamma(t)}\mmd t \right|
    &=\left|\int_{x-\beta}^{r}\frac{1}{i\gamma'(t)}\dfrac{\mmd}{\mmd t}(e^{i\gamma(t)})\mmd t \right|\cr
    &=\dfrac{1}{|\gamma'(r)|}+\dfrac{1}{|\gamma'(x-\beta)|}+\left| \frac{1}{\gamma'(r)}-\frac{1}{\gamma'(x-\beta)} \right|,
\end{split}
\end{equation*}
by integration by parts, as $\dfrac{\mmd}{\mmd t} (\dfrac{1}{\gamma'})$ does not change sign in $(0,\infty)$. Also, since $|\gamma'|$ is non-decreasing in $(0,\infty)$, $|\gamma'(r)|\geq |\gamma'(x-\beta)|$, we conclude that for all $x>\beta$,
\begin{equation*}
\begin{split}
        M_{\gamma}(f_{\beta})(x)\leq \frac{2}{(x-\beta)|\gamma'(x-\beta)|}.
\end{split}
\end{equation*}

\noindent Analogously, if $x<-\beta$,
\begin{equation*}
\begin{split}
    M_{\gamma}(f_{\beta})(x)&= \sup_{-\beta-x\leq r \leq \beta-x}\frac{1}{2r}\left| \int_{-\beta}^{x+r} e^{i\gamma(x-t)}\mmd t \right|\\
     &\leq\frac{1}{2(-x-\beta)}\sup_{-x-\beta\leq r \leq \beta-x}\left| \int_{-r}^{x+\beta} e^{i\gamma(t)}\mmd t \right|
\end{split}   
\end{equation*}
and
\begin{equation*}
\begin{split}
    \left| \int_{-r}^{x+\beta} e^{i\gamma(t)}\mmd t \right|
    \leq\dfrac{1}{|\gamma'(x+\beta)|}+\dfrac{1}{|\gamma'(-r)|}+\left| \frac{1}{\gamma'(x+\beta)}-\frac{1}{\gamma'(-r)} \right|\leq \frac{4}{|\gamma'(x+\beta)|},
\end{split}
\end{equation*}
since $-r\leq x+\beta<0$ implies that $|\gamma'(-r)|\geq|\gamma'(x+\beta)|$.

\vspace{0.2cm}
By the previous conclusions and by the fact that for all $x \in \mathbb{R}$, $M_{\gamma}(f_{\beta})(x)\leq ||f_{\beta}||_{\infty}=1$, we have that, for all $M>\max\{2\beta,2R\}$:
\begin{equation*}
\begin{split}
   ||M_{\gamma}(f_{\beta})||_{1} &\leq 2M+\int_{M}^{\infty} \frac{2}{(x-\beta)|\gamma'(x-\beta)|}\mmd x+\int_{-\infty}^{-M}\dfrac{2}{(|x|-\beta)|\gamma'(x+\beta)|}\mmd x\\
 &\leq 2M+2\int_{M}^{\infty} \frac{1}{(x/2)|\gamma'(x/2)|}\mmd x+2 \int_{-\infty}^{-M}\dfrac{1}{(|x|/2)|\gamma'(x/2)|}\mmd x\\
 &=2M+4\int_{|x|>M/2} \dfrac{1}{|x||\gamma'(x)|}\mmd x< \infty.
\end{split}
\end{equation*}

If one looks carefully at the proof above, one notices that the first and second conditions can be replaced by
    \begin{enumerate}
        \item $\exists T>0$ such that $|\gamma'|>0$ and it is non-decreasing in $(T, \infty)$.
        \item $\exists \tilde{T}<0$ such that 
        $|\gamma'|>0$ and it is non-increasing in $( -\infty, \tilde{T})$,
    \end{enumerate}
and the same conclusion still holds. We can therefore prove that oscillatory maximal functions associated to curves like $\gamma(t):=\sum_{j=1}^{m}c_j|t|^{d_j}$ where $d_1<d_2<...<d_m$ are real numbers, $d_m>1$ and $c_m\neq 0$, do map 
characteristic functions, and thus also step functions, into $L^1(\R).$ Since step functions are dense in $L^{1}$ and $M_{\gamma}f\in L^{1}$ for any step function $f$, one is tempted to conclude that $M_{\gamma}f\in L^{1}$ for all $f\in L^{1}$, 
although this is not true, as we cannot control the quantities $\{\frac{\|M_{\gamma}f\|_{1}}{\|f\|_{1}}: f \text{ non-zero step function}\}$ uniformly by some constant, and otherwise it would contradict Theorem \ref{thm:main_negative}. 

\end{remark}

\section{Proof of the boundedness results} 

\subsection{Preliminaries} Here and henceforth, we work with the class of functions 
\[
C_{p,l} =\{ f  \in \mathcal{S}'(\R) \colon \|(1+|x|^{l})f\|_{1}+\|f'\|_{p} < + \infty\}.
\]
This is a hybrid between a weighted $L^1$ and a usual Sobolev space on the real line. As such, it inherits many of the good properties of those spaces. The following two propositions
illustrate this fact.

\begin{proposition} \label{fact3}
Let $1\leq p\leq \infty$ and $l\in [0,\infty)$. Any $f\in C_{p,l}$ coincides almost everywhere with a locally absolutely continuous function.
\end{proposition}
\begin{proof}
In order to avoid confusion with the notation, let us temporarily denote by $Df$ the weak derivative of $f$ and $f'$ the derivative in the classical sense.

If $f\in C_{p,l}$, then $Df\in L^{p}(\R)$, in particular, $Df\in L^{1}_{\text{loc}}(\R)$ and we can define:
$$\tilde{f}(x)=\int_{0}^{x}Df(t)\mmd t.$$

Let $a<b$, $\eps>0$ and $\{[a_j,b_j]\}_{j}$ a finite or countably infinite collection of nonoverlapping subintervals of $[a,b]$.
Since $g:=Df\cdot \chi_{[a,b]}  \in L^{1}(\R)$ there exists a $\delta>0$ such that for $E$ measurable set with $|E|<\delta$ we have $\int_{E}|g|<\eps$, from this follows that if $\sum_{j}(b_j-a_j)<\delta$, then
$$\sum_{j}|\tilde{f}(b_j)-\tilde{f}(a_j)|\leq \int_{\cup(a_j,b_j)}|Df|=\int_{\cup(a_j,b_j)}|g|<\eps.$$
So, $\tilde{f}\in AC_{\text{loc}}(\R)$.
It holds then that $\tilde{f}$ is weakly differentiable with weak derivative equal to $Df$, and this implies that $\tilde{f}-f$ has identically zero weak derivative. Therefore, there exists a constant $c$ such that $f=c+\tilde{f}$ a.e. with $c+\tilde{f}\in AC_{\text{loc}}(\R)$, and $Df=D\tilde{f}=\tilde{f}'$ almost everywhere.
\end{proof}

\begin{proposition}
For any $1\leq p\leq \infty$ and $1\leq q\leq \infty$, we have
$$\|f\|_{q}\leq 2\|f\|_{C_{p,0}}\text{ , for all }f\in C_{p,0}.$$
\end{proposition}
\begin{proof}
Take $f\in C_{p,0}$ with $\|2f\|_{1}+\|f'\|_{p}=\|f\|_{C_{p,0}}=1$ and $f\in AC_{\text{loc}}(\R)$. We must show that $ \|f\|_{q}\leq 2$.

By the continuity of $f$, $E:=\{x\in \R\colon |f(x)|>1\}=\cup_{i}(a_i,b_i)$ is a countable union of disjoint open intervals. Also, $|E|\leq \int_{|f|>1}|f|\leq \|f\|_{1}< 1$, so we have that $\forall i$ and $\forall\, x\in (a_i,b_i)$
$$|f(x)|\leq |f(a_i)|+\int_{a_i}^{x}|f'(t)|\mmd t\leq 1+\|f'\|_{p}(x-a_i)^{1/p'}\leq 1+1\cdot|E|^{1/p'}\leq 2.$$

If $q=\infty$, we are done. If not,
$$\|f\|_{q}^{q}=\int_{|f|\leq 1}|f(t)|^{q}\mmd t+\int_{|f|>1}|f(t)|^{q}\mmd t
\leq \int_{|f|\leq1}|f|+2^{q-1}\int_{|f|>1}|f|\leq 1+2^{q-1}\leq 2^{q}.$$

\end{proof}

Before proving our main positive result, we start by stating some elementary Lemmas, which will help us throughout the proof. 

\begin{lemma}\label{llogl} Let $S \subset \R$ be a set with finite Lebesgue measure. It holds that 
\[
\int_{S} Mf(x) \, \mmd x \le |S| + C\int_{\R} |f(x)| \log(e+|f(x)|) \, \mmd x,
\]
where $M$ denotes the Hardy-Littlewood maximal operator.
\end{lemma}

\begin{proof} We use the layer-cake representation 
\[
\int_{S} Mf(x) \, \mmd x = \int_{0}^{\infty} |\{x \in S\colon Mf(x) > t\}| \, \mmd t
\]
and analyse the super level sets $A_t(S) := \{x \in S \colon Mf(x) > t\}:$ we split $|f| = |f| \cdot 1_{\{|f| > t/2\}} + |f| \cdot 1_{\{|f| \le t/2\}} =: f_1 + f_2.$ 
Clearly, $Mf_2(x) \le t/2,$ so we have that 
\[
A_t(S) \subset \{x \in S\colon Mf_1(x) > t/2\}.
\]
By the weak-type 1-1 inequality for the Hardy--Littlewood maximal operator, we get 
\[
|A_t(S)| \le C\frac{2}{t} \int_{\R} f_1(x) \, \mmd x = \frac{2C}{t} \int_{|f| > t/2} |f(x)| \, \mmd x. 
\]
An application of Fubini's theorem gives us then that
\begin{equation*}
 \begin{split}
\int_0^{\infty} |A_t(S)| \, \mmd t & \le |S| + \int_1^{\infty} |A_t(S)| \, \mmd t \\
& \le |S| + 4C \int_{\R} |f(x)| \log(e + |f(x)|) \, \mmd x.
\end{split} 
\end{equation*}
\end{proof}

\begin{lemma}\label{embedding}
Let $1\leq p\leq\infty $. There exists a constant $A>0$ such that for any $f\in C_{p,0}$ with $\|f\|_{C_{p,0}}\leq 2$ we have
$$\| |f|\log(e+|f|)\|_{1}\leq A.$$
\end{lemma}
\begin{proof}
By Fact \ref{fact3} we can assume that $f\in AC_{\text{loc}}(\R)$. In particular, $f$ is continuous.\\
For any $M>0$, we can write
$$\int_{\R}|f|\log(e+|f|)=\int_{|f|\leq M}|f|\log(e+|f|)+\int_{|f|>M}|f|\log(e+|f|).$$

For the first integral, we just use that $\|f\|_{1}\leq \frac{1}{2}\|f\|_{C_{p,0}}\leq 1$, so 
$$\int_{|f|\leq M}|f|\log(e+|f|)\leq \log(e+M)\|f\|_{1}\leq \log(e+M).$$

To estimate the second integral notice that by the continuity of $f$, the set $\{|f|>M\}$ is open in $\R$ so it is a countable union of disjoint open intervals $\{|f|>M\}=\cup_{j}(a_j,b_j)$.
If we take $M\geq1$, then
$$|\{|f|>M\}| \leq \int \frac{|f|}{M}\leq \frac{\|f\|_{1}}{M}\leq 1.$$
For any $j$, and $x\in(a_j,b_j)$ since $f\in AC_{\text{loc}}[a_j,b_j]$, we have 
$$|f(x)-f(a_j)|\leq\int_{a_j}^{x}|f'(t)|\mmd t\leq \|f'\|_{p}\cdot (x-a_j)^{1/p'}\leq \|f'\|_{p},$$
because $x-a_j\leq b_j-a_j\leq 1$, then $(x-a_j)^{1/p'}\leq 1$. Consequently,
$$|f(x)|\leq |f(a_j)|+|f(x)-f(a_j)|\leq M+ \|f'\|_{p}\leq M+2.$$
Since we have $|f|\leq M+2$ for any $x\in \{|f|>M\}$ we conclude that 
$$\int_{|f|>M}|f|\log(e+|f|)\leq \log(e+M+2)\int |f|\leq \log(e+M+2)$$
and 
$$\||f|\log(e+|f|)\|_{1}\leq \log(e+M)+\log(e+M+2).$$

Taking $M=1$, we see that $A=\log(e+1)+\log(e+3)$ works.
\end{proof}

Finally, we state a technical Lemma that will be key for the proof in the case of phases with polynomial factors. 

\begin{lemma}\label{serie}If $k\geq2$, then
$$\sum_{l=1}^{\infty}l\left|{k\choose l}\right|<\infty.$$
\end{lemma}
\begin{proof}
Notice that for any $l\geq \floor{k}+2$, the product $k(k-1)(k-2)...(k-l+1)$ starts with some positive factors $k,\,k-1, ...,k-\floor{k}$ and eventually its factors become negative, since $0\geq k-\floor{k}-1\geq k-\floor{k}-2,...\geq k-l-1$ .

Define $\delta=k-\floor{k}$, then $\delta\in [0,1)$ and exploring the previous observation:
\begin{equation*}
    \begin{split}
        \sum_{l\geq \floor{k}+2} l\left|{k\choose l}\right|=&
        \sum_{l\geq \floor{k}+2} \frac{1}{(l-1)!}
        |k(k-1)...(k-l+1)|\\
        =& \sum_{l\geq \floor{k}+2} \frac{1}{(l-1)!}|k(k-1)...\delta(\delta-1)(\delta -2)...(\delta+\floor{k}-l+1)|\\
        =& k(k-1)...(k-\floor{k})\sum_{l\geq \floor{k}+2} \frac{(1-\delta)(2-\delta)...(l-\floor{k}-1-\delta)}{(l-1)!}.
    \end{split}
\end{equation*}
Just calling $m=l-\floor{k}-1$ and $C_k=k(k-1)...(k-\floor{k})$ we get
\begin{equation*}
    \begin{split}
        \sum_{l\geq \floor{k}+2} l|{k\choose l}|=&
        =C_{k}\sum_{m\geq 1} \frac{(1-\delta)(2-\delta)...(m-\delta)}{(m+\floor{k})!}\\
        =&\sum_{m\geq 1}\frac{1-\delta}{1}\frac{2-\delta}{2}...\frac{m-\delta}{m}
        \frac{1}{(m+1)(m+2)...(m+\floor{k})}\\
        &\leq \sum_{m\geq 1} \frac{1}{(m+1)(m+2)...(m+\floor{k})}\leq \sum_{m\geq 1} \frac{1}{(m+1)(m+2)}<\infty.
    \end{split}
\end{equation*}
\end{proof}

\subsection{Proof of the main theorem} 

\begin{proof}[Proof of Theorem \ref{thm:main_positive}] First note that if $x\neq 0$ and $|t|<|x|$, 
$$\gamma(x,x-t)=\sum_{j=1}^{m}c_j(x)|x-t|^{d_j}=
\begin{cases}
\sum_{j=1}^{m}c_j(x)(x-t)^{d_j}, \text{ if }x>0. \\
\sum_{j=1}^{m}c_j(x)(t-x)^{d_j}, \text{ if }x<0.
\end{cases}$$
For example, if $x>0$, using Taylor expansions:
\begin{equation*}
    \begin{split}
     \gamma(x,x-t)=&\sum_{j=1}^{m}c_j(x)\sum_{l\geq 0} {d_j\choose l}(-t)^{l}x^{d_j-l}\\
     =&\sum_{j=1}^{m}c_j(x)x^{d_j}+\sum_{j=1}^{m}c_j(x) d_j(-t)x^{d_j-1}+\sum_{j=1}^{m}c_j(x)\sum_{l\geq 2}{d_j \choose l}(-t)^{l}x^{d_j-l}.
    \end{split}/
\end{equation*}
Then, for all $x\neq 0$, defining $$g(x;t):=f(t)\exp\left(i\sum_{j=1}^{m}c_j(x)\sum_{l\geq2}{d_j\choose l}(\pm t)^{l}|x|^{d_j-l}\right)$$
where we pick the sign $-$ if $x>0$ and $+$ if $x<0$, we have
\begin{equation}
    \begin{split}\label{conta}
        \frac{1}{|x|}|\int_{-|x|}^{|x|}&f(t)\exp(i\gamma(x,x-t))\mmd t|=\frac{1}{|x|}\left|\int_{-|x|}^{|x|}g(x;t)\exp\left(\pm i\sum_{j=1}^{m}c_j(x)d_j t|x|^{d_j-1}\right)\mmd t\right|\\
        \leq&\frac{1}{|x|}\frac{1}{|\sum_{j=1}^{m}c_j(x)d_j|x|^{d_j-1}|}
        \left\{ |f(x)|+|f(-x)|+\int_{-|x|}^{|x|}|\partial_t g(x;t)|\mmd t\right\}\\
        \leq& \frac{2}{d_m\|1/c_m\|_{\infty}^{-1}|x|^{d_m}}\left\{ |f(x)|+|f(-x)|+\int_{-|x|}^{|x|}|\partial_t g(x;t)|\mmd t\right\}
    \end{split}
\end{equation}
if $|x|\geq M$, for some suitable $M>1$, as $c_j \in L^{\infty}(\R), \, i=1,\dots,m-1.$

Now, let us see how to prove the result. Take $f\in C_{p,l}$ with $\|f\|_{C_{p,l}}=1$. Modifying $f$ in a set of measure zero, we can assume $f\in AC_{loc}(\R)$.

Pick $r$ measurable such that,
$$\frac{M_{\gamma}f(x)}{2}\leq \left|\dashint_{x-r(x)}^{x+r(x)}f(t)e^{i\gamma(x,x-t)}\mmd t\right|.$$
We can take for example $r(x):=\inf \left\{r>0\colon \left|\dashint_{-r}^{r} f(x-t)e^{i\gamma(x,t)} \mmd t\right| \ge \frac{M_{\gamma}f(x)}{2}\right\}$.

We will split the real line as the union of the following sets, for every $0<\varepsilon > 1$ 
\begin{itemize}
    \item $A_{1}:=[-M,M]$,
    \item $A_{2}:=\{x\notin [-M,M]\colon r(x)\leq |x|/2 \text{ and }M_{\gamma}f(x)\leq \frac{1}{|x|^{1+\eps}} \}$,
    \item $A_3:=\{x\notin [-M,M]\colon r(x)\leq |x|/2  \text{ and }M_{\gamma}f(x)>\frac{1}{|x|^{1+\eps}}\}$,
    \item $A_4:=\{x\notin [-M,M]\colon r(x)>|x|/2\}$.
\end{itemize}

\noindent\textit{Case 1: $A_1=[-M,M].$} For this interval, we use Lemmas \ref{llogl} and \ref{embedding}: We have 
\[
\int_{-1}^1 M_{\gamma}f(x) \, \mmd x \le \int_{-1}^1 Mf(x) \, \mmd x \le 2 + C \||f| \log(e+|f|)\|_1 \lesssim 1.
\]

\vspace{0,5cm}
\noindent\textit{Case 2:  $A_2 = \{x \in \R \setminus [-M,M]\colon r(x) \leq |x|/2 \text{ and } M_{\gamma}f(x) \leq 1/{|x|^{1+\eps}}\}$.} In this case, we bound directly 
\[
\int_{A_2} M_{\gamma}f(x) \, \mmd x \le 2 \int_1^{+\infty} \frac{\mmd x}{x^{1+\eps}} = \frac{2}{\eps}\lesssim_{\eps} 1.
\]

\vspace{0.5cm}
\noindent\textit{Case 3: $A_3 = \{x \in \R \setminus [-M,M]\colon r(x) \leq |x|/2 \text{ and } M_{\gamma}f(x) > {1}/{|x|^{1+\eps}}\}$.} In this case, we get immediately that

\[\frac{1}{2|x|^{1+\eps}}\leq \frac{M_{\gamma}(x)}{2}\leq 
\dashint_{x-r(x)}^{x+r(x)} |f(t)| \, \mmd t.
\]

\noindent As $r(x)\leq \frac{|x|}{2}$, it holds that for all $t\in (x-r(x),x+r(x)), \, |x|\leq 2|t|$. This implies that
\[
\frac{1}{2}\leq \dashint_{x-r(x)}^{x+r(x)} |f(t)| |x|^{1+\eps}\, \mmd t\leq 2^{1+\eps}\dashint_{x-r(x)}^{x+r(x)} |f(t)| |t|^{1+\eps}\, \mmd t\leq 4M(f(1+|t|^{1+\eps}))(x).
\]
Then $A_{3}\subset \{x\in \R\setminus[-1,1]\colon M((1+|t|^{1+\eps})f)(x)\geq \frac{1}{8}\}$. By the weak-type 1-1 bound for the Hardy--Littlewood 
maximal operator, we obtain again that 
\[
|A_3| \le 8C \cdot \|(1+|t|^{1+\eps}) f\|_1 \leq 8C,
\]
where we used that $\|f\|_{C_{p,1+\eps}}=1$. 
By Lemma \ref{llogl} again, we infer that 
\[
\int_{A_3} M_{\gamma}f(x) \leq \int_{A_3}Mf(x)\mmd x \, \mmd x \lesssim |A_3|+\|f \log(e+|f|)\|_{1}\lesssim 1. 
\]

\vspace{0.5cm}

\noindent\textit{Case 4: $A_4.$} For the last case, we need to split further into two other sets: 

\noindent\textit{Subcase 4.1: $A_{4,1} =  \{x \in A_4 \colon r(x) \ge 2|x|\}.$ } For this, we write the integral defining $M_{\gamma}$ as 
\begin{equation}\label{eq:threeintegrals}
 \int_{x-r(x)}^{-|x|} f(s) e^{i\gamma(x,x-s)} \, \mmd s + \int_{-|x|}^{|x|} f(s) e^{i\gamma(x,x-s)} \, \mmd s + \int_{|x|}^{x+r(x)}f(s) e^{i\gamma(x,x-s)} \, \mmd s.   
\end{equation}

The outer integrals are easier to handle. Indeed, there we bound the integrand pointwise by $|f(s)| \leq |f(s)|\cdot \left(\frac{|s|}{|x|}\right)^{\eps} $. Therefore, as $\|f|s|^{\eps}\|_{1}\leq \|(1+|s|^{1+\eps})f\|_{1}\leq 1$ 
both of the outermost integrals are bounded by $\frac{1}{|x|^{\eps}}$. This will be good enough since $\int_{|x|\geq M}\frac{1}{|x|^{1+\eps}}\mmd x<\infty$.

\vspace{5mm}
The inner integral has to be estimated a bit more carefully. We have that, for all $|x|>M$\\
\begin{equation}\label{eq:upper-bound}
    \begin{split}
        \frac{1}{|x|}|\int_{-|x|}^{|x|}&f(t)\exp(i\gamma(x,x-t))\mmd t|\\
        \lesssim & \frac{1}{|x|^{d_m}}\left\{ |f(x)|+|f(-x)|+\int_{-|x|}^{|x|}|\partial_t g(x;t)|\mmd t\right\}.\\
    \end{split}
\end{equation}

Recalling that $g(x;t):=f(t)\exp\left(i\sum_{j=1}^{m}c_j(x)\sum_{l\geq2}{d_j\choose l}(\pm t)^{l}|x|^{d_j-l}\right)$, then the integral which we need to estimate it 
$$\frac{1}{|x|^{d_m}}\int_{-|x|}^{|x|}|\partial_t g(x;t)|\mmd t,$$
which is bounded by 
\begin{equation}\label{eq:ugly-estimate}
    \begin{split}
          \frac{1}{|x|^{d_m}}\int_{-|x|}^{|x|} & |f'(t)|\mmd t+\frac{1}{|x|^{d_m}}\int_{-|x|}^{|x|}|f(t)|
        \sum_{j=1}^{m}|c_j(x)|\sum_{l\geq 2}l\left|{d_j\choose l}\right||t|^{l-1}|x|^{d_j-l}\mmd t\\
        \leq \frac{1}{|x|^{d_m}}\|f'\|_{p}(2|x|)^{1/p'} & +\frac{1}{|x|^{d_m}}\sum_{j=1}^{m}|c_j(x)|\sum_{l\geq 2}l\left|{d_j\choose l}\right||x|^{d_j-l} \int_{-|x|}^{|x|}|f(t)||t|^{\eps}|t|^{l-1-\eps}\mmd t\\
        \lesssim \frac{1}{|x|^{d_m-1/p'}}+ &
        \frac{1}{|x|^{d_m}}\sum_{j=1}^{m}|c_j(x)|\sum_{l\geq 2}l\left|{d_j\choose l}\right||x|^{d_j-1-\eps} \int_{-|x|}^{|x|}|f(t)||t|^{\eps}\mmd t\\
        \leq \frac{1}{|x|^{d_m-1/p'}}+ & 
        \frac{1}{|x|^{1+\eps}}\sum_{j=1}^{m}|c_j(x)|\sum_{l\geq 2}l\left|{d_j\choose l}\right|.
    \end{split}
\end{equation}
For all $1\leq j\leq m$, we have $d_j\geq 2 $, then using Lemma \ref{serie}, $\sum_{l\geq 2}l\left|{d_j\choose l}\right|<\infty$. From this, and from the fact that all the $c_j$ are bounded, it follows that
$$\frac{1}{|x|^{d_m}}\int_{-|x|}^{|x|}|g'(t)|\mmd t \lesssim
\frac{1}{|x|^{d_m-1/p'}}+\frac{1}{|x|^{1+\eps}}$$
\noindent and this is enough to conclude the analysis of this subcase, as the above expression is integrable for $|x| \ge M$ in the case where $d_m > 1/p' + 1,$ and putting together \eqref{eq:threeintegrals} and \eqref{eq:upper-bound} yields 
\[
\int_{A_{4,1}} M_{\gamma}f(x) \, \mmd x \lesssim_{p,s,\gamma} 1.
\]

\noindent\textit{Subcase 4.2: $A_{4,2} = A_4 \backslash A_{4,1}.$} In this set, the strategy we employ is still essentially the same as in Subcase 4.1. The difference in this case is that, instead of bounding $M_{\gamma}$ pointwise by three 
integrals as in \eqref{eq:threeintegrals}, we bound it by the sum of two: 

\begin{equation}\label{eq:upper-bound-2}
    \begin{split}
      M_{\gamma}f(x)\lesssim 
    & \frac{1}{|x|}\left\{\left|\int_{x-r(x)}^{x}f(s)e^{i\gamma(x,x-s)}\mmd s\right|+\int_{x}^{x+r(x)}|f(s)|\cdot\left(\frac{|s|}{|x|}\right)^{\eps}\mmd s\right\}\\
        \leq&\frac{1}{|x|}\left\{\left|\int_{x-r(x)}^{x}f(s)e^{i\gamma(x,x-s)}\mmd s\right|+\frac{1}{|x|^{\eps}}\right\}.
    \end{split}
\end{equation}

The integral in the upper bound in \ref{eq:upper-bound-2} is treated in the same way as in Subcase 4.1. Indeed, we bound 
\begin{equation*}
\begin{split}
      &\left|\int_{x-r(x)}^{x} f(s) e^{i\gamma(x,x-s)} \, \mmd s \right|=\left|\int_{x-r(x)}^{x} g(x;s) \left(\pm i\sum_{j=1}^{m}c_j(x)d_j s|x|^{d_j-1}\right) \, \mmd s \right| \\
 &\lesssim \frac{1}{|x|^{d_m}}\left\{|f(x)|+|f(x-r(x))|+\int_{x-r(x)}^{x}|\partial_s g(x;s)|\mmd s\right\},
\end{split}
\end{equation*}
for $|x| \ge M.$ Notice that since we are assuming that $f$ is locally absolutely continuous,
\begin{align*}
|f(x-r(x))| & \leq |f(x)|+|f(x-r(x))-f(x)|\leq |f(x)|+\int_{x-r(x)}^{x} |f'(s)|\mmd s \cr 
 & \le |f(x)| + (r(x))^{1/p'} \|f'\|_p \lesssim |f(x)| + |x|^{1/p'} \|f'\|_p. 
\end{align*}
The analysis of $\int_{x-r(x)}^{x}|\partial_s g(x;s)|\mmd s$ follows essentially the same ideas as in \eqref{eq:ugly-estimate}, and we therefore skip it. This implies, as in the previous subcase, that 
\[
\int_{A_{4,2}} M_{\gamma}f(x) \, \mmd x \lesssim_{p,s,\gamma} 1. 
\]
Putting the analysis of the four cases together, we conclude the proof of Theorem \ref{thm:main_positive}.
\end{proof}

\section{Further comments and remarks} 

\subsection{Different Weights and Theorem \ref{thm:main_positive}} 

For all the boundedness results that we saw in the previous section, we looked at subspaces of $L^{1}(\R)$ in that we had the existence of weak derivative belonging to some $L^{p}(\R)$ and some weighted integrability for the function, that is, $f$ times a certain weight  $\varphi_{\eps}(x)=1+|x|^{1+\eps}$ in $L^{1}$. One might ask whether we can consider different weights and ask for less integrability of the function
and still obtain the same result.

More precisely, for $\varphi_{\eps}(x)=1+|x|^{1+\eps}$, we defined $$ C_{p,1+\eps}=\{f\in L^{1}(\varphi_{\eps}(x)\mmd x)\colon \exists\text{ weak derivative }f'\in L^{p}(\R)\},$$ with the norm
$$\|f\|_{C_{p,1+\eps}}=\|f\cdot \varphi_{\eps}\|_{1}+\|f'\|_{p}.$$
Since $\varphi_{\eps} \geq 1$, $\|f\|_{1}\leq \|f\cdot \varphi_{\eps}\|_{1}$ and $C_{p,1+\eps}$ is a subspace of $L^{1}(\R)$.

More generally, for any $\varphi\geq 1$ we can consider the subspace of $L^{1}(\R)$ given by: $$ C_{p}^{\varphi}=\{f\in L^{1}(\varphi(x) \mmd x)\colon f \text{ is weakly differentiable and its weak derivative }f'\in L^{p}(\R)\}$$
with the norm
$\|f\|_{C_{p}^{\varphi}}=\|\varphi(\cdot)f\|_{1}+\|f'\|_{p}$.
\vspace{0.2cm}
By Theorem \ref{thm:main_positive} we have that
$$\|M_{\gamma}f\|_{1}\lesssim_{\gamma,\eps, p} \| f\cdot \varphi_{\eps} \|_{1}+\|f'\|_{p}=\|f\|_{C_{p}^{\varphi_{\eps}}},$$
whenever $\gamma$ is a polynomial expression as before. It is natural, therefore, to ask the following question
\begin{question}
Can we improve the weight $\varphi_{\eps}$?
\end{question}

The implicit constant that we get in the proof of the Theorem \ref{thm:main_positive} for the inequality above blows up if one sends $\eps \rightarrow 0$, so it seems that a more careful analysis is needed. The proof of Theorem \ref{thm:main_positive} can, 
however, be adapted to get improvements in the weight, in the sense of replacing $\varphi_{\eps}$ with a smaller function. Minor changes in the argument give us the following result.\\

\begin{theorem}

Let $\varphi:\R\rightarrow \R_{+}$ radial, $\varphi(t) \gtrsim 1+|t|$, $\varphi$ eventually non-decreasing,
$\varphi(2x)\lesssim \varphi(x)$ and 
$\int_{|x|>R}\frac{1}{\varphi(x)}\mmd x<\infty  $, some $R>0$. Then 
$$\|M_{\gamma}f\|_{1}\lesssim_{p,\gamma,\varphi} \|f\|_{C_{p}^{\varphi}}\text{, for all }k\in\Z_{\geq 2},\, 1\leq p<\infty.$$
In particular, the weights $\psi_{m}(x)=1+|x|(\log|x|)^{m}\chi_{|x|\geq 1}(x)$ are admissible for this Theorem whenever $m > 1.$ 
\end{theorem}

In fact, Case 1 in the proof of Theorem \ref{thm:main_positive} remain exactly the same. The set $A_2$ is replaced by the set where $M_{\gamma}(x) \leq 1/\varphi(x),$ and similarly for $A_3.$ In analogy, the integrability at infinity is used for Case 2, 
and the doubling condition on $\varphi$ for Case 3. Finally, we use that $\varphi(t) \gtrsim 1 + |t|$ in Case 4, as then we might replace any weighted integral 
$\int_{\R} |f(t)| \cdot |t|^{\alpha} \, \mmd t$ by $\|f \cdot \varphi\|_1.$ We leave the details for the interested reader. 

\subsection{Regularity of $M_{\gamma}$ and related endpoint questions} In comparison to the results by Tanaka \cite{Tanaka2002}, Aldaz and P\'erez-L\'azaro \cite{AldazPerezLazaro2006} and Kurka \cite{Kurka2015}, where it is proved that 
\[
\| (\mathcal{M}f)'\|_1 \le C \|f'\|_1,
\]
for $f \in BV(\R),$ where $\mathcal{M}$ stands for either the centered on uncentered Hardy--Littlewood maximal functions. As we have proven the exitence of non-trivial spaces that are mapped into $L^1$ by our oscillatory maximal functions $M_{\gamma},$ it 
is interesting to ask whether the same holds for the Sobolev space $W^{1,1}.$ 

\begin{question}\label{qu:sobolev} Let $\gamma: \R \times \R \to \R$ be as in the proof of Theorem \ref{thm:main_positive}. Does it hold that 
\[
M_{\gamma} : W^{1,1} \to W^{1,1} \text{ boundedly?}
\]
\end{question}

As partial reason in order to expect a positive answer to this question, we notice that, as $M_{\gamma}$ is a \emph{sublinear} operator, then, by the Haj\l{}asz-Onninen maximal principle \cite[Theorem~1]{HajlaszOnninen2004}, it holds that 
\[
M_{\gamma} : W^{1,p}(\R) \to W^{1,p}(\R) \text{ boundedly,}
\]
for $p>1.$ The proof of Theorem 1 in \cite{HajlaszOnninen2004} also proves that $M_{\gamma}$ maps $C_{p,1+\eps}$ into $C_{p,0},$ for any $\eps > 0$ and $p > 1.$ 

If true, we speculate that a positive answer to Question \ref{qu:sobolev} should incorporate the main ideas of Kurka's proof with additional decompositions of the real line, in order to 
get rid of the weights $\varphi_{\eps}(x) = 1+|x|^{1+\eps}$ in the statement of Theorem \ref{thm:main_positive}. 

\subsection{The $L \log L$ class and Theorem \ref{thm:main_positive}} If one looks carefully into the proof of Theorem \ref{thm:main_positive}, it is noticeable that there are only a couple of spots where any regularity on $f$ was demanded, additionally 
to that 
\[
f \log (e +|f|) \in L^1.
\]
Indeed, Case 3 in the proof of Theorem \ref{thm:main_positive} is the first time we need to employ the weighted condition $f \in L^1((1+|x|^{1+\eps})\,\mmd x),$ and the only occasion where we need to use that the derivative of $f$ belongs to a certain 
$L^p$ space is in Case 4 of the same proof, in order to translate oscillatory behaviour into decay. It is natural, therefore, to conjecture that these conditions are, in fact, superfluous. 

\begin{question} Is the operator $M_{\gamma}$ bounded from $L \log L$ to $L^1$, whenever $\gamma$ satisfies nontrivial curvature conditions, such as in Theorem \ref{thm:main_positive}? 
\end{question}

Although natural, this question might be hard to answer, as it would confirm that the infinity of the maximal functions $M_{\gamma}$ near infinity improves drastically with respect to the usual Hardy--Littlewood maximal function, which is only bounded from $L \log L$ to $L^1_{loc}.$ 
On the other hand, we believe that, with some additional effort, it should be possible to remove the regularity assumption on the function $f.$ In particular, we have not exploited completely the oscillatory nature of the phase, and a more careful decomposition and analysis, such as in
\cite{Lie20191,Lie20192,ZorinKranich2018}, should allow us to substitute the integration by parts in Case $4$ by some additional decay assumption on $f$.

\bibliography{bibliography}
\bibliographystyle{abbrv}

\end{document}